\documentclass[twoside,12pt]{article}
\usepackage[T1]{fontenc}
\usepackage[utf8]{inputenc}
\usepackage{lmodern}
\usepackage{mathrsfs}
\usepackage{amssymb, amsmath, mathtools, amsthm}
\usepackage{graphicx}
\usepackage{color}
\usepackage{float, caption}
\usepackage{enumitem}
\usepackage{booktabs}
\usepackage{array}
\usepackage{microtype}
\usepackage[top=2cm, bottom=2cm, left=2cm, right=2cm]{geometry}
\usepackage[colorlinks=true,linkcolor=blue,citecolor=blue,urlcolor=blue]{hyperref}

\IfFileExists{epsf.tex}{\input{epsf}}{}

\newcommand{\bd}{\begin{description}}
\newcommand{\ed}{\end{description}}
\newcommand{\bi}{\begin{itemize}}
\newcommand{\ei}{\end{itemize}}
\newcommand{\be}{\begin{enumerate}}
\newcommand{\ee}{\end{enumerate}}
\newcommand{\beq}{\begin{equation}}
\newcommand{\eeq}{\end{equation}}
\newcommand{\beqs}{\begin{eqnarray*}}
\newcommand{\eeqs}{\end{eqnarray*}}

\definecolor{DarkGreen}{rgb}{0.2, 0.6, 0.3}

\numberwithin{equation}{section}
\allowdisplaybreaks

\newtheorem{theorem}{Theorem}[section]

\newtheorem{lemma}[theorem]{Lemma}

\newtheorem{definition}[theorem]{Definition}

\newtheorem{proposition}[theorem]{Proposition}

\newcommand{\N}{\mathbb{N}}
\newcommand{\Z}{\mathbb{Z}}
\newcommand{\E}{\mathbb{E}}
\newcommand{\Prob}{\mathbb{P}}
\newcommand{\eps}{\varepsilon}
\newcolumntype{P}[1]{>{\raggedright\arraybackslash}p{#1}}

\hypersetup{
  pdftitle={Asymptotically optimal packing of fixed-cardinality arithmetic progressions},
  pdfauthor={Anonymous author},
  pdfsubject={Translated arithmetic-progression packing and rainbow matchings},
  pdfkeywords={arithmetic progression; translated-set packing; rainbow matching; random partition; LYM inequality}
}

\begin{document}

\title{\textbf{Solution to a conjecture of Alon, D\k{e}bski, Grytczuk and Przyby\l{}o on
fixed-cardinality arithmetic progressions\footnote{Supported by the National Science Foundation of China
(Nos. 12471329, 12401461 and 12061059).}}}
\author{Yaping Mao\footnote{Corresponding author: Faculty of Environment and Information
Sciences, Yokohama National University, 79-2 Tokiwadai, Hodogaya-ku,
Yokohama 240-8501, Japan. {Email: \tt maoyaping@ymail.com}}, \ \ Zhao
Wang\footnote{College of Science, China Jiliang University, Hangzhou
310018, China. {\tt
wangzhao@mail.bnu.edu.cn}}, \ \ 
Meiqin Wei\footnote{School of Science, Shanghai Maritime University, Shanghai 201306, China. {\tt weimeiqin8912@163.com}}, \ \ 
Gang Yang\footnote{Corresponding author. Faculty of Environment and Information Sciences, Yokohama National University, 79-2 Tokiwadai, Hodogaya-ku, Yokohama 240-8501, Japan. {\tt gangyang98@outlook.com}}}
\date{}
\maketitle

\begin{abstract}
Fix a positive integer $n$, and put $B_d=\{d,2d,\ldots,nd\}$.  Let $M_k(n)$
be the least integer $m$ for which one translate of each of
$B_1,\ldots,B_k$ can be placed pairwise disjointly in $[m]$.  We prove that,
for every $\eps\in(0,1)$ and all sufficiently large $k$, one has
$M_k(n)\le n\lceil(1+\eps)k\rceil$.  Since the trivial counting bound gives
$M_k(n)\ge nk$, it follows that $M_k(n)=(1+o(1))nk$ for every fixed $n$.
This confirms a conjecture of Alon, D\k{e}bski, Grytczuk and Przyby\l{}o on
prescribed-difference packings of fixed-cardinality arithmetic progressions.
\vspace{1mm}


\end{abstract}

\section{Introduction}

We write $[m]=\{1,2,\ldots,m\}$ for a positive integer $m$, and
$[a,b]=\{a,a+1,\ldots,b\}$ for integers $a\le b$.  In this paper the length
of an integer interval is the number of integer points it contains.  If
$A\subseteq\Z$ and $t\in\Z$, then $t+A=\{t+a:a\in A\}$ is called a translate
of $A$.

\begin{definition}
Let $\mathcal F=\{A_1,\ldots,A_s\}$ be a finite family of finite subsets of
$\Z$.  We say that $\mathcal F$ can be packed into $[m]$ if there are integers
$t_1,\ldots,t_s$ such that $t_i+A_i\subseteq[m]$ for every $i$ and the sets
$t_1+A_1,\ldots,t_s+A_s$ are pairwise disjoint.  The least such integer $m$ is
denoted by $m(\mathcal F)$.
\end{definition}

The problem belongs to a natural circle of questions about packing translated
copies of prescribed finite sets.  The additional arithmetic feature here is
that the sets are initial pieces of arithmetic progressions, and their common
differences are prescribed rather than chosen freely.  Thus one has to place
many sets whose shapes become longer and more rigid as the difference grows.
This is the source of the main difficulty.

Alon, D\k{e}bski, Grytczuk and Przyby\l{}o \cite{ADGP} initiated a systematic
study of this parameter for the progressions $D_d=\{d,2d,3d,\ldots\}$,
where $d\in\N$.  They considered two related models.  In the bounded-diameter
model one takes $A_d=D_d\cap[n]=\{d,2d,\ldots,\lfloor n/d\rfloor d\}$; all
sets then lie inside $[n]$, but their sizes depend on $d$.  In the
equal-cardinality model one takes $B_d=D_d\cap[nd]=\{d,2d,\ldots,nd\}$; all
sets then have exactly $n$ elements, but the diameter of $B_d$ grows with
$d$.

For $D\subseteq\N$, let $\mathcal F_D(n)=\{A_d:d\in D\}$ and
$\mathcal G_D(n)=\{B_d:d\in D\}$, and put $m_D(n)=m(\mathcal F_D(n))$ and
$M_D(n)=m(\mathcal G_D(n))$.  In particular, $m_k(n)=m_{[k]}(n)$,
$m(n)=m_n(n)$, $M_k(n)=M_{[k]}(n)$, and $M(n)=M_n(n)$.  We keep this notation
in order to match \cite{ADGP}.  The rest of the paper concerns only the
quantity $M_k(n)$ with $n$ fixed and $k$ tending to infinity.  Equivalently,
$M_k(n)$ is the least $m$ for which one can choose integers $t_1,\ldots,t_k$
so that $t_d+B_d\subseteq[m]$ for all $d\in[k]$, and these $k$ translated
sets are pairwise disjoint.

The two models have different asymptotic scales.  In the bounded-diameter
model, Alon et al. proved
\[
 \left(\frac43-o(1)\right)\frac{n^{3/2}}{\log n}
 \le m(n)\le
 \left(\frac53+o(1)\right)\frac{n^{3/2}}{\log n}.
\]
In the equal-cardinality model with $k=n$, they proved
\[
 \left(\frac16-o(1)\right)\frac{n^3}{\log n}
 \le M(n)\le
 (0.526+o(1))\frac{n^3}{\log n}.
\]
Their Theorems~4 and~5 also contain estimates in the range $k<n$, but those
estimates are not needed below.  The part of their work relevant here is the
following fixed-$n$ bound.  The lower bound
\begin{equation}\label{eq:lower-bound}
        M_k(n)\ge nk
\end{equation}
is immediate, since every packing contains $k$ disjoint sets of size $n$.
Alon et al. proved that, for every fixed $n$, there is $k_0(n)$ such that
$M_k(n)\le 3nk$ for all $k\ge k_0(n)$.  Their proof first packs almost all
progressions in the cyclic group $\mathbb Z_{nk}$ using the near-perfect
matching machinery of R\"odl, Frankl, Pippenger and Spencer
\cite{FranklRodl,PippengerSpencer,Rodl}.  After the cycle is cut open, the
main part of the packing lies in an ordinary interval of length less than
$2nk$, and the remaining progressions are placed in a second interval of
asymptotic length $nk$.  This explains the factor $3$ in their upper bound and
also shows why removing that extra factor is not just a matter of tightening a
constant.

The fixed-$n$ question was formulated by Alon, D\k{e}bski, Grytczuk and
Przyby\l{}o as the conjecture that $M_k(n)=(1+o(1))nk$ for every fixed
positive integer $n$.  The theorem below proves this statement.

\begin{theorem}\label{thm:main}
For every fixed positive integer $n$ and every $\eps\in(0,1)$, there exists
$k_0=k_0(n,\eps)$ such that $M_k(n)\le n\lceil(1+\eps)k\rceil$ for every
$k\ge k_0$.  Consequently, $M_k(n)=(1+o(1))nk$ as $k\to\infty$.
\end{theorem}

Let us place this theorem beside the exact small cases.  For $n=1$, the sets
$B_d$ are singletons, and $M_k(1)=k$.  For $n=2$, a packing is a family of
pairs whose distances are $1,2,\ldots,k$.  A packing in $[2k]$ is precisely a
Skolem sequence of order $k$: each symbol $d\in[k]$ appears twice in a sequence
of length $2k$, with the two occurrences at distance $d$.  If one allows one
unused position, one obtains a hooked Skolem sequence and a packing in
$[2k+1]$.  The classical existence theorems of Skolem and O'Keefe
\cite{OKeefe1961, Skolem1957} give
\[
 M_k(2)=
 \begin{cases}
  2k,   & k\equiv0,1\pmod4,\\
  2k+1, & k\equiv2,3\pmod4.
 \end{cases}
\]
For $n\ge3$, the exact picture is much less clear.  Already for $n=3$ it is
not known in general when an interval of length $3k$ can be partitioned into
three-term arithmetic progressions with distinct differences; asking that the
differences be exactly $1,\ldots,k$ is stronger still.

The proof of Theorem~\ref{thm:main} is built from three ingredients.  First, we
construct a random partition of an integer $L$ into parts of size at most $k$,
with the expected number of parts of size $d$ equal to $L/(kd)$.  This
harmonic-looking identity is important because, after a suitable normalization,
it gives total mass $1/d$ to intervals of length $d$.  Secondly, we order the
parts at random.  The LYM inequality then shows that, for a fixed point and a
fixed size $d$, the chance of seeing a $d$-part there is small.  This gives a
fractional packing with uniform total load and with no single difference too
concentrated at any point.  Finally, an interval of length $d$ is replaced by
$d$ disjoint $n$-term arithmetic progressions of common difference $d$.  After
choosing several possible translates for each difference, the theorem of
Delcourt and Postle on full rainbow matchings selects one translate of every
difference without intersections.

All logarithms are natural.  Unless explicitly stated otherwise, asymptotic
notation is used with $n$ and $\eps$ fixed and $k\to\infty$.  Constants hidden
in $O_n(\cdot)$ may depend on $n$, and constants hidden in
$O_{n,\eps}(\cdot)$ may depend on both $n$ and $\eps$.

\section{Auxiliary tools and harmonic interval weights}

We first recall the hypergraph language used in the final step.  A
multi-hypergraph is allowed to contain the same edge more than once; all
degrees below count such repeated edges with multiplicity.  It is
$r$-bounded if every edge has at most $r$ vertices and $r$-uniform if every
edge has exactly $r$ vertices.  For a multi-hypergraph $H$ and a vertex $v$,
let $\deg_H(v)$ be the number of edges containing $v$, and let $\Delta(H)$ be
the maximum degree.  For distinct vertices $u$ and $v$, the codegree
$\deg_H(u,v)$ is the number of edges containing both.

When the edges are colored, the color-degree $\deg_c(v)$ is the number of
edges of color $c$ containing $v$.  A matching is a family of pairwise disjoint
edges.  A matching is rainbow if no two of its edges have the same color; it is
full rainbow if it contains one edge of every color.

\begin{theorem}[Theorem~2.11 \cite{DP}]\label{thm:DP}
For every integer $r\ge2$ and every real $\beta>0$, there exist a threshold
$D_0$ and a constant $\alpha>0$, depending only on $r$ and $\beta$, with the
following property.  If $D\ge D_0$ and $H$ is an $r$-bounded edge-colored
multi-hypergraph such that $\Delta(H)\le D$, every color appears on at least
$(1+D^{-\alpha})D$ edges, every pair of distinct vertices has codegree at most
$D^{1-\beta}$, and $\deg_c(v)\le D^{1-\beta}$ for every vertex $v$ and color
$c$, then $H$ has a full rainbow matching.
\end{theorem}

The theorem permits parallel edges and a non-proper coloring.  Both features
are convenient, because later two different choices may lead to the same
underlying translate.

We shall use two standard Chernoff bounds.  If $X$ is a sum of independent
Bernoulli random variables with mean $\mu$, then for every $t\ge0$,
\begin{equation}\label{eq:chernoff-additive}
        \Prob(X\ge\mu+t)
        \le \exp\!\left(-\frac{t^2}{2(\mu+t/3)}\right).
\end{equation}
Also, whenever $A\ge e\mu$,
\begin{equation}\label{eq:chernoff-large}
        \Prob(X\ge A)\le\left(\frac{e\mu}{A}\right)^A .
\end{equation}
These forms remain true for non-identically distributed Bernoulli variables;
see, for example, \cite{AlonSpencer}.  We also use the LYM inequality
\cite{Lubell}: if $\mathcal A$ is an antichain in the Boolean lattice
$2^{[r]}$, then
\begin{equation}\label{eq:LYM}
        \sum_{S\in\mathcal A}\binom{r}{|S|}^{-1}\le1.
\end{equation}

We now turn to the fractional construction.  A partition of $L$ means a
multiset of positive integers whose sum is $L$.  Equal parts are regarded as
distinct occurrences when they are labelled or ordered.  A composition of $L$
is an ordered list of positive integers summing to $L$.

\begin{lemma}\label{lem:harmonic-partition}
For every pair of integers $L\ge k\ge1$, there exists a random partition
$\Pi_{L,k}$ of $L$ into parts of size at most $k$ such that, if $N_d$ denotes
the number of parts of size $d$, then $\E N_d=L/(kd)$ for every $d\in[k]$.
\end{lemma}

\begin{proof}
The construction is by induction on $L$.  If $L=k$, choose a uniformly random
permutation of $[k]$ and take its cycle lengths as the parts.  A uniform
permutation has expected number $1/d$ of cycles of length $d$, which is
$L/(kd)$ in this case.

Assume now that $L>k$, and suppose that the desired distribution has already
been defined for all smaller first arguments.  Choose a uniformly random
permutation of $[L]$, conditioned on not being a single $L$-cycle.  Each
cycle of length at most $k$ is kept as one part.  Each cycle of length
$m>k$ is replaced, independently of all other such cycles, by a sample from
$\Pi_{m,k}$.  The conditioning excludes the only cycle of length $L$, so every
recursive call has first argument smaller than $L$.

Let $C_m$ be the number of top-level cycles of length $m$.  In an
unconditioned uniform permutation of $[L]$, $\E C_m=1/m$.  For $m<L$, the
excluded event has $C_m=0$ and probability $1/L$.  Therefore, in the
conditioned distribution, $\E C_m=L/(m(L-1))$ for $m<L$.  A final part of size
$d\le k$ comes either from a top-level $d$-cycle or from the recursive
replacement of a longer cycle.  Hence the induction hypothesis gives
\[
 \E N_d
 =\frac{L}{d(L-1)}
   +\sum_{m=k+1}^{L-1}\frac{L}{m(L-1)}\cdot\frac{m}{kd}
 =\frac{L}{kd}.
\]
The induction is complete.
\end{proof}

The identity in Lemma~\ref{lem:harmonic-partition} is compatible with the
obvious conservation of length, since $\sum_{d=1}^k d\,\E N_d=L$.  We also
need a mild lower-tail estimate for the total number of parts.

\begin{lemma}\label{lem:many-parts}
There is an absolute constant $c_0>0$ such that, whenever $L>k\ge2$ and $R$
is the number of final parts in $\Pi_{L,k}$, one has
$\Prob(R<c_0\log k)=O(k^{-1/4})$.  One may take $c_0=1/(4\log2)$, and the
implied constant is absolute.
\end{lemma}

\begin{proof}
Let $C$ be the number of cycles in the top-level permutation.  Every top-level
cycle produces at least one final part, so $R\ge C$.  For an unconditioned
uniform permutation $\sigma\in S_L$, the cycle-count generating function is
$\E z^{C(\sigma)}=z(z+1)\cdots(z+L-1)/L!$.  Putting $z=1/2$ gives $\E 2^{-C(\sigma)}=\binom{2L}{L}/4^L\le L^{-1/2}$.
Conditioning on the event that $\sigma$ is not a single $L$-cycle changes this
expectation by a factor at most $1/(1-1/L)\le2$.  Thus, in the conditioned
distribution, $\E 2^{-C}\le2L^{-1/2}$.  With $c_0=1/(4\log2)$, Markov's
inequality gives
\[
 \Prob(C<c_0\log k)
 \le 2^{c_0\log k}\E 2^{-C}
 \le 2k^{1/4}L^{-1/2}
 \le 2k^{-1/4},
\]
since $L\ge k$.  The same estimate holds for $R$, because $R\ge C$.
\end{proof}

The next lemma explains why a random ordering of many parts spreads each
individual part size rather evenly.

\begin{lemma}\label{lem:start-position}
Let $b_1,\ldots,b_R$ be labelled positive integers with sum $L$, and place the
corresponding blocks in a uniformly random order along $[L]$.  Fix a block
$i$.  The probability that this block starts at any prescribed point is at
most $1/R$.  If $b_i=d$, then the probability that it covers any prescribed
point of $[L]$ is at most $d/R$.
\end{lemma}

\begin{proof}
Fix a possible starting point $s$ and put $t=s-1$.  A set
$S\subseteq[R]\setminus\{i\}$ can be exactly the set of blocks placed before
block $i$ only if $\sum_{j\in S}b_j=t$.  Since all $b_j$ are positive, all such
sets $S$ form an antichain.  For a fixed such $S$, the probability that
precisely the blocks in $S$ precede block $i$ is
$|S|!(R-1-|S|)!/R!$, or equivalently $1/(R\binom{R-1}{|S|})$.  Summing this
quantity over the antichain and applying \eqref{eq:LYM} gives the bound
$1/R$.  A block of length $d$ can cover a fixed point from at most $d$ starting
positions, and the second assertion follows by a union bound.
\end{proof}

For $d\in[k]$ and $1\le s\le L-d+1$, write $I_{d,s}=[s,s+d-1]$.

\begin{proposition}\label{prop:fractional-interval}
There is an absolute constant $C>0$ such that, whenever $k<L\le2k$ and $k$ is
sufficiently large, there are nonnegative weights $q_{d,s}$ satisfying
\begin{align}
 \sum_{s=1}^{L-d+1}q_{d,s}&=\frac1d
        &&(d\in[k]),\label{eq:q-color}\\
 \sum_{d=1}^k\ \sum_{s:\,x\in I_{d,s}}q_{d,s}&=\frac{k}{L}
        &&(x\in[L]),\label{eq:q-load}\\
 \sum_{s:\,x\in I_{d,s}}q_{d,s}&\le\frac{C}{\log k}
        &&((d,x)\in[k]\times[L]).\label{eq:q-diffuse}
\end{align}
\end{proposition}

\begin{proof}
Choose a partition from Lemma~\ref{lem:harmonic-partition}, label all part
occurrences, and then put the labelled parts in a uniformly random order.  This
produces a random composition of $L$, or equivalently a random decomposition of
$[L]$ into consecutive blocks.

Let $Q_{d,s}$ be the probability that $I_{d,s}$ is one of these blocks, and set
$q_{d,s}=(k/L)Q_{d,s}$.  Summing over $s$ counts the expected number of blocks
of length $d$, so Lemma~\ref{lem:harmonic-partition} gives
$\sum_s q_{d,s}=(k/L)\E N_d=1/d$, which proves \eqref{eq:q-color}.  In every
realization each point $x\in[L]$ lies in exactly one block.  Taking
expectations and multiplying by $k/L$ gives \eqref{eq:q-load}.

It remains to prove \eqref{eq:q-diffuse}.  Fix $d$ and $x$.  Conditional on
the unordered labelled partition, Lemma~\ref{lem:start-position}, applied to
each labelled $d$-part, gives the conditional bound
$\Prob(x\hbox{ lies in a $d$-block}\mid\Pi_{L,k})\le dN_d/R$.  On the event
$R\ge c_0\log k$ this is at most $dN_d/(c_0\log k)$, while outside that event
we use the trivial bound $1$.  Hence Lemmas~\ref{lem:harmonic-partition} and
\ref{lem:many-parts} imply
\[
 \Prob(x\hbox{ lies in a $d$-block})
 \le \frac{d\E N_d}{c_0\log k}+\Prob(R<c_0\log k)
 =\frac{L}{kc_0\log k}+O(k^{-1/4})
 \le \frac{C}{\log k},
\]
where $L\le2k$ and $C$ is enlarged if necessary.  Multiplying by $k/L$ gives
\eqref{eq:q-diffuse}.
\end{proof}

The three identities have distinct roles.  Formula \eqref{eq:q-color} gives
intervals of length $d$ total mass $1/d$.  Formula \eqref{eq:q-load} says that
every point of $[L]$ carries exactly the same total mass, namely $k/L$.
Formula \eqref{eq:q-diffuse} says that the contribution of any one fixed
length $d$ at a single point is small.  The first two statements are exact;
only the third is asymptotic.

\section{From intervals to arithmetic progressions}

We now convert the interval weights into weights on translates of the sets
$B_d$.  Fix $n\in\N$, let $k<L\le2k$, and take the weights from
Proposition~\ref{prop:fractional-interval}.  Think of every point of $[L]$ as
being replaced by a block of $n$ consecutive integer points, so that the new
ambient interval is $[nL]$.

For $d\in[k]$, $1\le s\le L-d+1$, and $a\in[d]$, define
\begin{equation}\label{eq:edge-definition}
        E_{d,s,a}=\{n(s-1)+a+jd:0\le j<n\}.
\end{equation}
This is a translate of $B_d$, because $E_{d,s,a}=(n(s-1)+a-d)+B_d$.  For fixed
$d$ and $s$, the $d$ sets in \eqref{eq:edge-definition} are pairwise disjoint
and partition a consecutive interval:
\begin{equation}\label{eq:bundle-partition}
        \bigsqcup_{a=1}^d E_{d,s,a}
        =[n(s-1)+1,\,n(s+d-1)].
\end{equation}
Indeed, after subtracting $n(s-1)$, the left-hand side is
$\{a+jd:1\le a\le d,\ 0\le j<n\}=[nd]$.

For each $d\in[k]$, let $\Omega_d=\{(s,a):1\le s\le L-d+1,\ a\in[d]\}$.  If
$\omega=(s,a)\in\Omega_d$, write $E_d(\omega)=E_{d,s,a}$ and set
$\nu_d(\omega)=q_{d,s}$.  Different labels may occasionally determine the same
underlying subset of $[nL]$; this causes no difficulty, since the labels are
only used to describe the probability distribution.

\begin{proposition}\label{prop:fractional-AP}
For every $d\in[k]$, the function $\nu_d$ is a probability distribution on
$\Omega_d$, and every $E_d(\omega)$ is a translate of $B_d$ contained in
$[nL]$.  If $p_d(v)=\Prob_{\omega\sim\nu_d}(v\in E_d(\omega))$, then for every
$v\in[nL]$,
\begin{align}
        \sum_{d=1}^k p_d(v)&=\frac{k}{L},\label{eq:actual-load}\\
        p_d(v)&\le\frac{C}{\log k}\qquad(d\in[k]).\label{eq:actual-diffuse}
\end{align}
Moreover, for any two distinct vertices $u,v\in[nL]$, at most $n-1$ values of
$d\in[k]$ can have positive probability of producing an edge containing both
$u$ and $v$.
\end{proposition}

\begin{proof}
For a fixed difference $d$, the total mass on $\Omega_d$ is
\[
        \sum_{(s,a)\in\Omega_d}\nu_d(s,a)
        =\sum_{s=1}^{L-d+1}\sum_{a=1}^d q_{d,s}
        =d\sum_s q_{d,s}=1,
\]
by \eqref{eq:q-color}.  Thus $\nu_d$ is a probability distribution.  The
translate property follows from \eqref{eq:edge-definition}, and containment in
$[nL]$ follows from \eqref{eq:bundle-partition}.

Fix $v\in[nL]$, and let $x=\lceil v/n\rceil\in[L]$.  For fixed $d$ and $s$,
the partition identity \eqref{eq:bundle-partition} shows that exactly one of
the $d$ possible values of $a$ gives a set containing $v$ if $x\in I_{d,s}$,
and none does if $x\notin I_{d,s}$.  Hence $p_d(v)=\sum_{s:\,x\in I_{d,s}}q_{d,s}$.
The two displayed conclusions now follow from \eqref{eq:q-load} and
\eqref{eq:q-diffuse}.

Finally, suppose that a translate of $B_d$ contains two distinct vertices $u$
and $v$.  Two terms of an $n$-term progression with common difference $d$ then
differ by $jd$ for some $j\in\{1,\ldots,n-1\}$; in other words,
$|u-v|=jd$.  For each fixed $j$, this equation determines at most one integer
$d$.  Hence at most $n-1$ differences can occur.
\end{proof}

The identity \eqref{eq:actual-load} also has the correct total size.  Summing
it over all $nL$ vertices gives
$\sum_{v\in[nL]}\sum_{d=1}^k p_d(v)=nL\cdot k/L=nk$, which is exactly the
expected total number of vertices obtained by choosing one $n$-term progression
for each of the $k$ differences.

\section{Random choice of translates and the proof of the theorem}

We now choose finitely many possible translates for each difference.  The
following lemma says that, for large $k$, the choices may be made so that the
resulting colored hypergraph satisfies the hypotheses of Theorem~\ref{thm:DP}.
The exponents $4$ and $7/8$ have no special significance; they simply leave a
comfortable gap between the expected local degrees and the allowed threshold.

\begin{lemma}\label{lem:sampling}
Fix $n\ge2$ and $\eps\in(0,1)$.  For all sufficiently large $k$, there is an
$n$-uniform edge-colored multi-hypergraph $H$ on $V(H)=[nL]$, where
$L=\lceil(1+\eps)k\rceil$, with color set $[k]$, such that, if
$\eta=\eps/(1+\eps)$, $T=\lceil\log^4 k\rceil$, and
$D=\lceil(1-\eta/2)T\rceil$, then the following hold:
\begin{enumerate}[label=\textup{(\roman*)},leftmargin=2.4em]
\item every color class has exactly $T$ edges;
\item $\Delta(H)\le D$;
\item every pair of distinct vertices has codegree at most $D^{7/8}$;
\item $\deg_d(v)\le D^{7/8}$ for every $v\in V(H)$ and every $d\in[k]$.
\end{enumerate}
Every edge of color $d$ is a translate of $B_d$ contained in $[nL]$.
\end{lemma}

\begin{proof}
For large $k$, we have $k<L\le2k$, so Proposition~\ref{prop:fractional-AP}
applies.  For each color $d\in[k]$, independently choose $T$ labels
$\omega\in\Omega_d$ according to $\nu_d$, allowing repetitions.  Each chosen
label gives the edge $E_d(\omega)$ with color $d$.  If the same underlying set
appears more than once, it is kept with multiplicity.  This gives an
$n$-uniform edge-colored multi-hypergraph in which every color class has size
exactly $T$.

First consider vertex degrees.  For a fixed vertex $v$, all indicators for the
chosen edges are independent, and \eqref{eq:actual-load} gives
\[
        \E\deg_H(v)=T\sum_{d=1}^k p_d(v)=T\frac{k}{L}
        \le \frac{T}{1+\eps}=(1-\eta)T.
\]
For sufficiently large $k$, $\eta T/2\ge1$ and $D\le T$.  If
$\mu_v=\E\deg_H(v)$, then $D-\mu_v$ lies between $\eta T/2$ and $T$.  Applying
\eqref{eq:chernoff-additive} gives $\Prob(\deg_H(v)>D)\le\exp(-c_\eps T)$ for
some $c_\eps>0$.  Since $|V(H)|=nL=O_{n,\eps}(k)$, the probability that the
maximum-degree condition fails is $o(1)$.

Now put $\rho=C/\log k$, where $C$ is the constant from
Proposition~\ref{prop:fractional-AP}.  For fixed $d$ and $v$, the color-degree
$\deg_d(v)$ is binomial with mean at most $\rho T=O(\log^3 k)$, by
\eqref{eq:actual-diffuse}.  For distinct $u,v$, let
$r_d(u,v)=\Prob_{\omega\sim\nu_d}(\{u,v\}\subseteq E_d(\omega))$.  By
Proposition~\ref{prop:fractional-AP}, at most $n-1$ of the numbers
$r_d(u,v)$ are nonzero, and each is at most $p_d(u)\le\rho$.  Hence the
codegree of $u$ and $v$ has mean at most $(n-1)CT/\log k=O_n(\log^3 k)$.

On the other hand, $D^{7/8}=\Theta_\eps(\log^{7/2}k)$.  Thus the threshold is
larger than the relevant means by a factor of order $\log^{1/2}k$.  Let
$A=\lfloor D^{7/8}\rfloor+1$.  For every color-degree and every pair codegree
under consideration, the mean $\mu$ is at most $C_{n,\eps}\log^3k$, and so
$A\ge e\mu$ once $k$ is large.  Formula \eqref{eq:chernoff-large} gives,
uniformly over all these random variables,
\[
        \Prob(X>D^{7/8})
        \le \exp\!\left(-\Omega_{n,\eps}(\log^{7/2}k\log\log k)\right).
\]
There are only $k|V(H)|+\binom{|V(H)|}{2}=O_{n,\eps}(k^2)$ relevant
color--vertex pairs and unordered vertex pairs.  A union bound shows that the
color-degree and pair-codegree conditions fail with probability $o(1)$.  Hence
with positive probability all four stated properties hold simultaneously.
\end{proof}

\begin{proof}[Proof of Theorem~\ref{thm:main}]
If $n=1$, then every $B_d$ is a singleton, so $M_k(1)=k$.  We may assume
$n\ge2$.

Fix $\eps\in(0,1)$ and take the multi-hypergraph $H$ supplied by
Lemma~\ref{lem:sampling}.  We apply Theorem~\ref{thm:DP} with $r=n$ and
$\beta=1/8$.  Let $\alpha>0$ and $D_0$ be the corresponding constants.  Since
$D=\Theta_\eps(\log^4 k)$, the inequality $D\ge D_0$ holds for all sufficiently
large $k$.  Lemma~\ref{lem:sampling} gives $\Delta(H)\le D$, and it also gives
that every pair codegree and every color-degree is at most
$D^{7/8}=D^{1-\beta}$.

It remains to check the lower bound on the number of edges in each color
class.  Each color has exactly $T$ edges, while
$D/T=1-\eta/2+O(1/T)$.  Hence $T/D-1$ tends to $\eta/(2-\eta)>0$.  Since
$D^{-\alpha}$ tends to zero, for all sufficiently large $k$ we have
$T\ge(1+D^{-\alpha})D$.  Thus all hypotheses of Theorem~\ref{thm:DP} are
satisfied.

Theorem~\ref{thm:DP} gives a full rainbow matching in $H$.  This matching
contains one edge of every color $d\in[k]$.  The chosen edge of color $d$ is a
translate of $B_d$ contained in $[nL]$, and the chosen edges are pairwise
disjoint because they form a matching.  Therefore
$M_k(n)\le nL=n\lceil(1+\eps)k\rceil$ for all sufficiently large $k$.

Finally, $M_k(n)\le(1+\eps)nk+n$.  After division by $nk$ and passage to the
limit as $k\to\infty$, we obtain
$\limsup_{k\to\infty}M_k(n)/(nk)\le1+\eps$.  Since $\eps\in(0,1)$ was
arbitrary and \eqref{eq:lower-bound} gives the opposite lower bound, it follows
that $M_k(n)=(1+o(1))nk$.
\end{proof}

We finish with two comments.  The proof above is asymptotic in an essential
way: it leaves a fixed positive amount of unused space while the rainbow
matching is found, and removes this slack only after $k$ tends to infinity.
Thus it does not answer the exact tiling question for $n\ge3$.  It would be
very interesting to find, for each fixed $n$, an explicit function
$f_n(k)=o(k)$ such that $M_k(n)\le nk+f_n(k)$ for all sufficiently large $k$.
A more arithmetical rounding argument might also shed light on the exact pairs
$(n,k)$ for which $M_k(n)=nk$.

The construction also suggests a useful general viewpoint.  One wants a
fractional choice for every color with total vertex load bounded away from
$1$, with each single color spread thinly over the vertices, and with only
boundedly many colors capable of containing a fixed pair of vertices.  In the
present problem the harmonic partition gives the first two properties, and the
simple relation $|u-v|=jd$ gives the last one.

\medskip
\noindent {\bf Data availability.} No data were used for the research described in this article.

\end{document}